\documentclass[12pt,reqno]{amsart}

\usepackage[latin1]{inputenc}
\usepackage{amsmath}
\usepackage{amsfonts}
\usepackage{amssymb}
\usepackage{enumerate}
\usepackage{subfigure}
\usepackage{amssymb,amsmath,amsthm,amscd,epsf,latexsym,verbatim,graphicx,amsfonts}

\usepackage{amscd}
\usepackage{amsmath}
\usepackage{amssymb}
\usepackage{amsthm}
\usepackage{epsf}
\usepackage{latexsym}
\usepackage{verbatim}

\theoremstyle{plain}

\newtheorem{theorem}{Theorem}[section]

\newtheorem{prop}[theorem]{Proposition}
\theoremstyle{definition}
\newtheorem{conj}{Conjecture}

\theoremstyle{remark}

\newcommand{\bbR}{\mathbb{R}}
\newcommand{\bbC}{\mathbb{C}}

\newcommand{\bbD}{\mathbb{D}}
\newcommand{\bbN}{\mathbb{N}}

\newcommand{\mco}{\mathcal{O}}

\newcommand{\eitheta}{e^{i\theta}}

\newcommand{\nri}{n\rightarrow\infty}
\newcommand{\bnri}{N\rightarrow\infty}

\DeclareMathOperator*{\Real}{Re}

\begin{document}
\title[ ] {Torsional Rigidity and Bergman Analytic Content of Simply Connected Regions}

\bibliographystyle{plain}

\thanks{  }


\maketitle




\begin{center}
\textbf{Matthew Fleeman $\&$ Brian Simanek}\\
\end{center}


\begin{abstract}
We exploit the equality of Bergman analytic content and torsional rigidity of a simply connected domain to develop a new method for calculating these quantities.  This method is particularly suitable for the case when the region in question is a polygon.  A large number of examples are computed in explicit detail to demonstrate the utility of our ideas.
\end{abstract}

\vspace{8mm}


\smallskip


\section{Introduction}\label{intro}

Let $\Omega\subseteq\bbC$ be a bounded and simply connected Jordan domain.  Associated to such a region is its \textit{torsional rigidity}, which quantifies the resistance to twisting of a cylindrical beam with the given cross-section.  We denote the torsional rigidity of $\Omega$ by $\rho(\Omega)$.  There are several ways to formulate this quantity mathematically, but there is a particularly useful formula for simply connected regions that is given in \cite{PS}, namely
\begin{align}\label{rhodef}
\rho(\Omega):=\sup_{u\in C^1_0(\bar{\Omega})}\frac{4\left(\int_{\Omega}u(z)dA(z)\right)^2}{\int_{\Omega}|\nabla u(z)|^2dA(z)},
\end{align}
where $dA$ denotes area measure on $\Omega$ and $C^1_0(\bar{\Omega})$ denotes the set of all differentiable functions on $\Omega$ that vanish on the boundary of $\Omega$.  The supremum in this definition is in fact a maximum, and we will call any function for which the maximum is attained a \textit{stress function} for the region $\Omega$.  It is known that the solution to the boundary value problem
\begin{align}\label{bproblem}
\begin{cases}
\Delta v=-2,\qquad\qquad z\in\Omega\\
v\big|_{\partial\Omega}=0
\end{cases}
\end{align}
is a stress function (see \cite[page 88]{PS}), which we will denote by $\nu(z;\Omega)$ (or simply $\nu(z)$ if the region is clear).  With this notation, one has (see \cite[page 88]{PS})
\begin{align}\label{nofrac}
\rho(\Omega)=2\int_{\Omega}\nu(z;\Omega)dA(z).
\end{align}
It is easy to see that if $u(z)$ solves the Dirichlet Problem:
\begin{align*}
\begin{cases}
\Delta u=0,\qquad\qquad z\in\Omega\\
u\big|_{\partial\Omega}=\frac{|z|^2}{2},
\end{cases}
\end{align*}
then the stress function is given by $\nu(z;\Omega)=u(z)-\frac{|z|^2}{2}$.

Here are some basic facts about the torsional rigidity of simply connected Jordan regions, all of which follow easily from our discussion so far:
\begin{itemize}
\item for any $c\in\bbC$, $\rho(\Omega+c)=\rho(\Omega)$,
\item for any $r\in\bbC$, $\rho(r\Omega)=|r|^4\rho(\Omega)$,
\item if $\Omega_1$ and $\Omega_2$ are simply connected and $\Omega_1\subseteq\Omega_2$, then $\rho(\Omega_1)\leq\rho(\Omega_2)$,
\item if $\bbD=\{z:|z|<1\}$, then $\rho(\bbD)=\pi/2$.
\end{itemize}
This last equality follows from the fact that the boundary value problem (\ref{bproblem}) has explicit solution $(1-|z|^2)/2$.

If we consider the supremum in (\ref{rhodef}), the extremal function must have an average value on $\Omega$ that is as large as possible subject to the constraint that the gradient must also have a small average value.  This means the extremal function $\nu$ is $0$ on the boundary of $\Omega$ and gets to be as large as possible without growing too quickly.  Thus, domains with large torsional rigidity must have as much area as possible that is far away from the boundary.  It is this intuition that would lead one to conjecture that disks have maximal torsional rigidity among regions with fixed area.  This turns out to be correct and is a result known as Saint-Venant's inequality (see \cite{Pol} and \cite[page 121]{PS}).

\smallskip

Another property of $\Omega$ that we will be interested in is its \textit{Bergman analytic content}, which we define as in \cite{FK} to be the distance from $\bar{z}$ to the space of analytic functions in $L^2(\Omega,dA)$.  One motivation for the study of this object comes from operator theory.  Indeed, let $A^2(\Omega)$ be the Bergman Space, that is, the space of all analytic functions on $\Omega$ that lie in $L^2(\Omega,dA)$.  If $\psi$ is a conformal bijection of $\bbD$ with $\Omega$, then one can consider the operator $T_{\psi}:A^2(\bbD)\rightarrow A^2(\bbD)$ given by $T_{\psi}f=\psi f$.  The operator norm of the commutator $[T_{\psi}^*,T_{\psi}]$ is equal to
\[
\sup_{g\in A^2_1(\Omega)}\inf_{f\in A^2(\Omega)}\int_{\Omega}|\bar{z}g-f|^2dA(z),
\]
where $A^2_1(\Omega)$ is the unit ball in $A^2(\Omega)$ (see \cite[Proof of Theorem 2]{FK1}).  Therefore, if $\Omega$ has area $1$, then the square of the Bergman analytic content of $\Omega$ gives a lower bound on the norm of the commutator $[T_{\psi}^*,T_{\psi}]$.  By Putnam's Inequality (see \cite[Theorem 2.1]{PMBook}), this norm lower bounds the area of the spectrum of $T_{\psi}$.

If we imagine an arc $\Gamma$, then under suitable smoothness hypotheses on $\Gamma$, there exists a function $\phi(z)$ that is analytic in a neighborhood of $\Gamma$ and such that $\phi(z)=\bar{z}$ when $z\in\Gamma$.  Of course, $\phi(z)$ is the Schwarz function of $\Gamma$.  If $\Omega$ is very well approximated by $\Gamma$, then we expect $\phi(z)$ to be close to $\bar{z}$ on all of $\Omega$, and hence the Bergman analytic content of $\Omega$ would be small in this case.  Conversely, we would expect the Bergman analytic content to be large for a domain $\Omega$ that cannot be well approximated by a curve.  In other words, curves with large interior relative to their boundary should have large Bergman analytic content.  This would lead one to conjecture that among regions with fixed area, the one with largest Bergman analytic content is the disk.  This turns out to be correct and it was proven in \cite{FK}.

\smallskip

The definition (\ref{rhodef}) defines torsional rigidity as a maximization problem, so lower bounds can be obtained by trial and error.  The extremal problem defining Bergman analytic content is a minimization problem, so upper bounds can be obtained by trial and error.  It was shown in \cite{FL} that the square of the Bergman analytic content of $\Omega$ is equal to $\rho(\Omega)$ when $\Omega$ is simply connected (as it is in the case we are considering), so in this case one can obtain both upper and lower bounds by trial and error.

There are several methods available for estimating the torsional rigidity of a Jordan region $\Omega$, some of which are discussed in \cite{Sok}.  For instance, if one can write down an explicit Taylor series for a conformal bijection between the unit disk and the region, then one can use this series to write down the stress function or compute the torsional rigidity (see \cite[pages 97 $\&$ 120]{PS} and \cite[Section 81]{Sok}).  One can similarly calculate the torsional rigidity of $\Omega$ if one has precise knowledge of the Dirichlet spectrum for the region (see \cite[page 106]{PS}).  Another approach requires the computation of the expected lifetime of a Brownian Motion (see \cite[Equations 1.8 and 1.11]{BVdbC} and \cite{HMP}) and another approach considers (\ref{rhodef}) over the space of web functions (see \cite{Avk,CFG,Makai}).  These methods are often very difficult to employ because the necessary information is only available for a small number of specific regions.

One can also use the geometry of the region $\Omega$ to estimate the torsional rigidity.  Examples of estimates derived in this way can be found in \cite{Avk,Makai} and can be quite powerful.  However, these bounds do not produce a sequence of approximants that converge to the actual value, so the best result one can obtain from such estimates is an order of magnitude calculation.  It is our goal to develop a new method for calculating torsional rigidity that is based on approximation theory in the Bergman space.  This method requires the ability to calculate the area moments of a region, which is often a manageable task, even for complicated regions.  Furthermore, by sending the order of approximation to infinity, the resulting estimates will converge to the true torsional rigidity of the region.

The method we propose is especially well suited to the case when $\Omega$ is a polygon.  This will allow us to pursue certain extremal problems related to torsional rigidity.  There is a substantial interest in extremal problems associated to torsional rigidity that are analogous to isoperimetric problems (see \cite{BFL,Lipton,OR,Salak,SZ,vdBBV}).  It has been conjectured for nearly seventy years that the $n$-gon with area $1$ having maximal torsional rigidity is the regular $n$-gon of area $1$ (see \cite{Pol}), though this conjecture remains unproven for $n\geq5$.  The extremal problem that we will tackle in this work is that of finding the right triangle with area $1$ and maximal torsional rigidity.  We conjecture that the answer is the isosceles right triangle and we verify this conjecture to within an error of approximately one half of one percent.  

\smallskip

The rest of the paper is organized as follows.  After a brief review of some important facts about orthogonal polynomials in Section \ref{op},  we discuss some methods for calculating the torsional rigidity of a simply connected Jordan region $\Omega$ in Section \ref{methods}.  One method we discuss is based on using the area moments of the region to calculate torsional rigidity using orthogonal polynomials.  Another method we discuss is based on function theory in the unit disk.  This second method is not new, but we present some new details and expose its relationship to Bergman analytic content.  In Section \ref{egs} we present some additional examples that yield particularly elegant calculations.

\section{Orthogonal Polynomials}\label{op}

One of our main tools for calculation in the next section will be orthogonal polynomials, so we will recall some important notions here, especially in the setting of area measure and measures on the unit circle.  The theory of orthogonal polynomials in these settings is very rich and deep, so we will focus our attention on specific facts that will be relevant for our later calculations.  For more detailed information, we refer the reader to the references \cite{Chihara,Nevai,Szego}.

\subsection{Bergman Polynomials}\label{bergpoly}

Given a bounded Jordan region $\Omega$, the Bergman polynomials $\{p_n\}_{n=0}^{\infty}$ are orthonormal with respect to area measure on $\Omega$.  Their asymptotic properties (as the degree becomes large) are often determined by the smoothness of the boundary.  Let $\gamma:\bbC\setminus\overline{\Omega}\rightarrow\bbC\setminus\overline{\bbD}$ be the conformal bijection that satisfies $\gamma(\infty)=\infty$ and $\gamma'(\infty)>0$.  Let $\eta$ be the inverse map to $\gamma$.  If $\eta$ extends to be univalent in the exterior of the circle of radius $r<1$, then we say that $\Omega$ has analytic boundary and for every $s> r$ we may define $\Omega^s$ to be the exterior of the image of the circle of radius $s$ under the map $\eta$.  Carleman showed in \cite{Carleman} that for any $s>r$ it holds that
\begin{align}\label{pasy}
p_n(z)=\sqrt{\frac{n+1}{\pi}}\gamma(z)^n\gamma'(z)(1+\varepsilon_n(z)),\qquad\quad z\in\overline{\Omega^s},
\end{align}
where $\varepsilon_n(z)$ converges to $0$ exponentially quickly and uniformly on $\overline{\Omega^s}$ as $\nri$.  Under weaker smoothness assumptions on $\partial\Omega$, one can obtain a similar result with slower decay in the error term (see \cite[Theorems 1.1 $\&$ 1.2]{Suetin}).

\subsection{The Unit Circle}\label{uc}

Orthogonal polynomials on the unit circle are an especially nice class of orthogonal polynomials because of the Szeg\H{o} recursion.  If $\{\Phi_n\}_{n=0}^{\infty}$ denotes the sequence of monic polynomials orthogonal with respect to some probability measure $\mu$ having infinite support in the unit circle, then the Szeg\H{o} recursion states that there is a sequence $\{\alpha_n\}_{j=0}^{\infty}\in\bbD^\bbN$ such that
\[
\Phi_{n+1}(z)=z\Phi_n(z)-\bar{\alpha}_{n}\Phi_n^*(z),
\]
where $\Phi_n^*(z)=z^n\overline{\Phi_n(1/\bar{z})}$ (see \cite[Theorem 1.5.2]{OPUC1}).  From this relation, it is clear that $\alpha_n=-\overline{\Phi_{n+1}(0)}$.  Verblunsky's Theorem states that the converse is also true, namely that any sequence in $\bbD^\bbN$ determines an infinitely supported probability measure on the unit circle (see \cite[Section 1.7]{OPUC1}).  The coefficients $\{\alpha_n\}_{n=0}^{\infty}$ corresponding to the measure $\mu$ are known by a variety of names in the literature, including Verblunsky coefficients, Geronimus coefficients, and reflection parameters.

The Szeg\H{o} recursion can be inverted to give the Inverse Szeg\H{o} recursion
\[
\Phi_n(z)=\frac{\Phi_{n+1}(z)+\bar{\alpha}_n\Phi_n^*(z)}{z(1-|\alpha_n|^2)}
\]
(see \cite[Theorem 1.5.4]{OPUC1}).  One consequence of this relation is that if we know the degree $n$ monic orthogonal polynomial for the measure $\mu$, then we can determine the degree $m$ monic orthogonal polynomial for $\mu$ for every $m=0,1,\ldots,n-1$.

\section{Methods for Calculating and Estimating $\rho(\Omega)$}\label{methods}

In this section we will discuss a two methods for the calculation and estimation of the torsional rigidity, stress function, or Bergman projection of $\bar{z}$ for a Jordan region $\Omega$.  Recall that $\rho(\Omega)$ denotes the torsional rigidity of $\Omega$ and $\nu(z;\Omega)$ denotes the stress function that solves the boundary value problem (\ref{bproblem}).  Also, we will denote by $Q(z)$ the projection of $\bar{z}$ to the Bergman space of $\Omega$ and we will let $\psi:\bbD\rightarrow\Omega$ be any conformal bijection.

Some of the results that we present below are in the form of numerics that indicate the utility of our methods.  All calculations in this section were done on a desktop computer using Mathematica version 10.

\subsection{The Moments Approach}\label{moment}

As mentioned in Section \ref{intro}, many methods for calculating $\rho(\Omega)$ require some particular knowledge such as the conformal bijection $\psi$ or the Dirichlet spectrum of the region $\Omega$.  In this section we will discuss an approach that allows us to calculate $Q(z)$ and $\rho(\Omega)$ through knowledge of the moments of the area measure on $\Omega$.  This is an especially useful method to have available when $\Omega$ is a polygon, which has easily computable moments, but for which the Dirichlet spectrum and conformal bijection with the disk are hard to write down explicitly.  To do this calculation, we let $\{p_n(z;\Omega)\}_{n\geq0}$ be the sequence of Bergman orthonormal polynomials for the region $\Omega$ and let $\{P_n(z;\Omega)\}_{n\geq0}$ be the sequence of monic Bergman polynomials for $\Omega$.  Define
\[
c_{i,j}=\langle z^i,z^j\rangle=\int_{\Omega}z^i\bar{z}^jdA(z)
\]
It is easy to show that if $n\geq1$, then
\[
R_n(z;\Omega):=\begin{vmatrix}
c_{0,0} & c_{1,0} & c_{2,0} & \cdots & c_{n,0}\\
c_{0,1} & c_{1,1} & c_{2,1} & \cdots & c_{n,1}\\
\vdots & \vdots & \vdots & \ddots & \vdots \\
c_{0,n-1} & c_{1,n-1} & c_{2,n-1} & \cdots & c_{n,n-1}\\
1 & z & z^2 & \cdots & z^n
\end{vmatrix}
=\sigma_nP_n(z;\Omega)
\]
for some constant $\sigma_n>0$. 
Therefore, we can calculate the full sequence $\{p_n(z;\Omega)\}_{n\geq0}$ from knowledge of the area moments of $\Omega$.

By \cite[Theorem 2]{Farrell} we know that $\{p_n(z;\Omega)\}_{n\geq0}$ is an orthonormal basis for the Bergman space of $\Omega$, so we can write
\[
Q(z)=\sum_{n=0}^{\infty}\langle\bar{w},p_n(w;\Omega)\rangle p_n(z;\Omega)=\sum_{n=0}^{\infty}\langle1,wp_n(w;\Omega)\rangle p_n(z;\Omega),\qquad\qquad z\in\Omega
\]
  The sum converges because for any fixed $z\in\Omega$, the sequence $\{p_n(z;\Omega)\}_{n\geq0}$ is in $\ell^2(\bbN_0)$ (see \cite{Sima}).  We can estimate $Q$ (and hence $\rho$) by truncating this sum after only finitely many terms.  For every $N\in\bbN_0$, define
\begin{align*}
Q_N(z)&:=\sum_{n=0}^{N}\langle\bar{w},p_n(w;\Omega)\rangle p_n(z;\Omega)\\
\rho_N(\Omega)&:=\int_{\Omega}|z|^2dA(z)-\int_{\Omega}|Q_N(z)|^2dA(z).
\end{align*}
Clearly $\rho(\Omega)\leq\rho_N(\Omega)$ for each $N$, so this method will always yield upper estimates of the torsional rigidity.

One can see that
\[
\rho(\Omega)\leq\rho_N(\Omega)= c_{1,1}-\sum_{n=0}^N|\langle\bar{w},p_n(w;\Omega)\rangle|^2
\]
and
\[
|\langle\bar{w},p_n(w;\Omega)\rangle|^2=\frac{\left|\langle\bar{w},R_n\rangle\right|^2}{\|R_n\|^2}=\frac{\left|\langle\bar{w},R_n\rangle\right|^2}{\sigma_n\langle w^n,R_n\rangle}.
\]
Each factor in this expression can be expressed in terms of determinants.  Indeed, if $n\geq1$, then
\begin{align*}
\langle\bar{w},R_n\rangle=\begin{vmatrix}
c_{0,0} & c_{0,1} & c_{0,2} & \cdots & c_{0,n}\\
c_{1,0} & c_{1,1} & c_{1,2} & \cdots & c_{1,n}\\
\vdots & \vdots & \vdots & \ddots & \vdots \\
c_{n-1,0} & c_{n-1,1} & c_{n-1,2} & \cdots & c_{n-1,n}\\
c_{0,1} & c_{0,2} & c_{0,3} & \cdots & c_{0,n+1}
\end{vmatrix}
\end{align*}
\begin{align*}
\sigma_n=\begin{vmatrix}
c_{0,0} & c_{1,0} & c_{2,0} & \cdots & c_{n-1,0}\\
c_{0,1} & c_{1,1} & c_{2,1} & \cdots & c_{n-1,1}\\
\vdots & \vdots & \vdots & \ddots & \vdots \\
c_{0,n-2} & c_{1,n-2} & c_{2,n-2} & \cdots & c_{n-1,n-2}\\
c_{0,n-1} & c_{1,n-1} & c_{2,n-1} & \cdots & c_{n-1,n-1}
\end{vmatrix}
\end{align*}
\begin{align*}
\langle w^n,R_n\rangle=\begin{vmatrix}
c_{0,0} & c_{0,1} & c_{0,2} & \cdots & c_{0,n}\\
c_{1,0} & c_{1,1} & c_{1,2} & \cdots & c_{1,n}\\
\vdots & \vdots & \vdots & \ddots & \vdots \\
c_{n-1,0} & c_{n-1,1} & c_{n-1,2} & \cdots & c_{n-1,n}\\
c_{n,0} & c_{n,1} & c_{n,2} & \cdots & c_{n,n}
\end{vmatrix}
\end{align*}
Therefore, if we have convenient expressions for the moments $c_{i,j}$ of a region, then we can estimate the torsional rigidity by calculating determinants of specific square matrices.

For a general region, the formulas involved in these estimates become unwieldy as $\bnri$, but we can make some explicit estimates for small $N$.  Our estimates will involve the moments of area, which we define as
\[
I_{mn}=\int_{\Omega}x^my^ndxdy,\qquad\qquad m,n\in\bbN_0.
\]
As an example, we work through the calculations to show the best possible approximation to $\bar{z}$ using a polynomial of degree at most $1$ or $2$.  The estimate for polynomials of degree at most $1$, which we now present, is an improvement of the estimate given in \cite[page 111]{DW}.  Since the quantities we are interested in behave trivially under translation, we may always assume without loss of generality that the centroid of $\Omega$ is $0$. 

\begin{theorem}\label{deg1}
Let $\Omega$ be a simply connected and bounded region in $\bbC$ whose centroid is equal to $0$.  Then the projection of $\bar{z}$ onto the polynomials of degree at most $1$ in the Bergman space is given by $\alpha z$, where
\[
\alpha=\frac{c_{0,2}}{c_{1,1}}.
\]
Consequently
\[
\rho(\Omega)\leq\rho_1(\Omega)=4\frac{I_{20}I_{02}-I_{11}^2}{I_{20}+I_{02}}
\]
\end{theorem}

\begin{proof}
The assumption that the centroid is $0$ tells us that
\[
p_0(z;\Omega)=|\Omega|^{-1/2}\qquad\qquad \mbox{ and }\qquad\qquad p_1(z;\Omega)=\frac{z}{\sqrt{c_{1,1}}}
\]
Therefore,
\[
Q_1(z)=\frac{c_{0,2}}{c_{1,1}}z
\]
We then calculate
\begin{align*}
\rho(\Omega)\leq\rho_1(\Omega)&=\int_{\Omega}|\bar{z}-Q_1(z)|^2dA(z)=\int_{\Omega}|\bar{z}|^2dA(z)-\int_{\Omega}\left|\frac{c_{0,2}}{\sqrt{c_{1,1}}}p_1(z)\right|^2dA(z)\\
&=c_{1,1}-\frac{|c_{0,2}|^2}{c_{1,1}}=\frac{4I_{20}I_{02}-4I_{11}^2}{I_{20}+I_{02}}.
\end{align*}
\end{proof}

\noindent\textit{Remark.}  Notice that the conclusion of Theorem 1.2 resembles Cauchy's 1829 estimate (see \cite[Equation 5]{Avk}).

\smallskip

\noindent\textit{Remark.}  It is worth pointing out that the expression $I_{20}+I_{02}$ that appears in the denominator of this estimate is the polar moment of inertia of $\Omega$.

\medskip

We will state a comparable result for $\rho_2(\Omega)$, but first we will make some simplifying assumptions.  In Theorem \ref{deg1}, we used the assumption that the centroid of $\Omega$ is zero so that $c_{1,0}=I_{10}=I_{01}=0$.  Since the centroid of a region is rotation invariant, we can also choose the orientation of the region $\Omega$ in a way that works to our advantage. The Dominated Convergence Theorem implies that $I_{mn}$ varies continuously as we rotate the region and it is clear that $I_{mn}$ changes sign if we rotate $\Omega$ by $\pi$ and $m+n$ is odd.  Therefore, the Intermediate Value Theorem tells us that we can choose the orientation of $\Omega$ so that $I_{21}=0$.  This assumption will simplify our formulas in the following theorem.

\begin{theorem}\label{deg2}
Let $\Omega$ be a simply connected Jordan region in $\bbC$ that whose centroid is equal to $0$ and which satisfies $I_{21}=0$.  Then the projection of $\bar{z}$ onto the polynomials of degree at most $2$ in the Bergman space is given by
\[
\frac{c_{0,2}}{c_{1,1}}z+\frac{c_{0,0}c_{1,1}}{c_{0,0}c_{1,1}c_{2,2}-c_{0,0}|c_{2,1}|^2-c_{1,1}|c_{2,0}|^2}\left(c_{0,3}-\frac{c_{1,2}c_{0,2}}{c_{1,1}}\right)\left(z^2-\frac{c_{2,1}}{c_{1,1}}z-\frac{c_{2,0}}{c_{0,0}}\right)
\]
Consequently
\begin{align*}
&\rho_2(\Omega)=4\left[\frac{I_{20}I_{02}-I_{11}^2}{I_{20}+I_{02}}-\frac{I_{00}[\left(I_{02}\left(I_{30}-I_{12}\right)-2I_{20}I_{12}+I_{11}I_{03}\right)^{2}+\left(I_{20}I_{03}+I_{11}(I_{30}+I_{12})\right)^{2}]}{(I_{20}+I_{02})^2\left[I_{00}\left(I_{40}+2I_{22}+I_{04}-\frac{(I_{12}+I_{30})^2+I_{03}^2}{I_{20}+I_{02}}\right)-(I_{20}-I_{02})^2-4I_{11}^2\right]}\right]
\end{align*}
\end{theorem}

\begin{proof}
We calculate
\[
p_2(z;\Omega)=\sqrt{\frac{c_{0,0}c_{1,1}}{c_{1,1}c_{2,2}-|c_{2,1}|^2-c_{1,1}|c_{2,0}|^2}}\left(z^2-\frac{c_{2,1}}{c_{1,1}}z-\frac{c_{2,0}}{c_{0,0}}\right).
\]
We then have
\begin{align*}
Q_2(z)&=\langle\bar{z},p_1\rangle p_1+\langle\bar{z},p_2\rangle p_2\\
&=\frac{c_{0,2}}{\sqrt{c_{1,1}}}p_1(z;\Omega)+\sqrt{\frac{c_{0,0}c_{1,1}}{c_{0,0}c_{1,1}c_{2,2}-c_{0,0}|c_{2,1}|^2-c_{1,1}|c_{2,0}|^2}}\left(c_{0,3}-\frac{c_{1,2}c_{0,2}}{c_{1,1}}\right)p_2(z;\Omega)
\end{align*}
Therefore,
\[
\rho_2(\Omega)=c_{1,1}-\frac{|c_{0,2}|^2}{c_{1,1}}-\left(\frac{c_{0,0}c_{1,1}}{c_{0,0}c_{1,1}c_{2,2}-c_{0,0}|c_{2,1}|^2-c_{1,1}|c_{2,0}|^2}\right)\left|c_{0,3}-\frac{c_{1,2}c_{0,2}}{c_{1,1}}\right|^2
\]
The rest of the proof follows from algebraic manipulation and the fact that $I_{21}=0$.
\end{proof}

Let us consider some specific examples to which we can apply this method and in particular Theorems \ref{deg1} and \ref{deg2}.

\subsubsection{The Rectangle}  Consider the rectangle $(-a/2,a/2)\times(-b/2,b/2)$.  In this case, the moments of area $I_{mn}$ are especially easy to calculate, as are the area moments $c_{i,j}$.  If we apply Theorem \ref{deg1} to this rectangle and combine the result with \cite[page 99]{PS}, then we get
\[
\frac{a^3b^3}{4(a^2+b^2)}\leq\rho(\Omega)\leq\frac{a^3b^3}{3(a^2+b^2)}.
\]
Theorem \ref{deg2} gives the same estimate because the region $\Omega$ in this case is symmetric with respect to the real axis and the imaginary axis, so $I_{mn}=0$ whenever $m$ or $n$ is odd.  Recall the formula on \cite[page 108]{PS}:
\begin{align}\label{rectangle}
\rho(\Omega)&=\frac{256a^3b^3}{\pi^6}\sum_{k,j=0}^{\infty}\frac{1}{(2j+1)^2(2k+1)^2\left((2j+1)^2a^2+(2k+1)^2b^2\right)}\\
\nonumber&\qquad\quad\approx\frac{0.266281a^3b^3}{a^2+b^2}+\frac{256a^3b^3}{\pi^6}\sum_{{k,j=0}\atop{k+j>0}}^{\infty}\frac{1}{(2j+1)^2(2k+1)^2\left((2j+1)^2a^2+(2k+1)^2b^2\right)}
\end{align}
(see also \cite[Section 38]{Sok}).  We see that by using the best approximation to $\bar{z}$ among polynomials with degree at most $1$, we nearly recover the first term in this series.

Since we know the scaling properties of torsional rigidity, we may assume without loss of generality that the area of the rectangle is $1$, so $b=1/a$.  Let us denote the corresponding rectangle by $\Omega(a)$.  We employed the moments approach to find an upper bound on the torsional rigidity of $\Omega(a)$ using polynomials of degree at most $12$.  Define the function $R(a):(0,\infty)\rightarrow(0,\infty)$ by using the sum (\ref{rectangle}) with $b=1/a$ and letting $j$ and $k$ run from $0$ to $85$.  It is clear that $R(a)$ is an underestimate for $\rho(\Omega(a))$, while we have already observed that $\rho_{12}(\Omega(a))$ is an overestimate for $\rho(\Omega(a))$.  Figure 1 shows a plot of $\rho_{12}(\Omega(a))/R(a)-1$ for values of $a$ between $0$ and $10$.  Notice that the error is smaller than one half of one percent.

\begin{figure}[h!]\label{rectpic}
  \centering
    \includegraphics[width=0.55\textwidth]{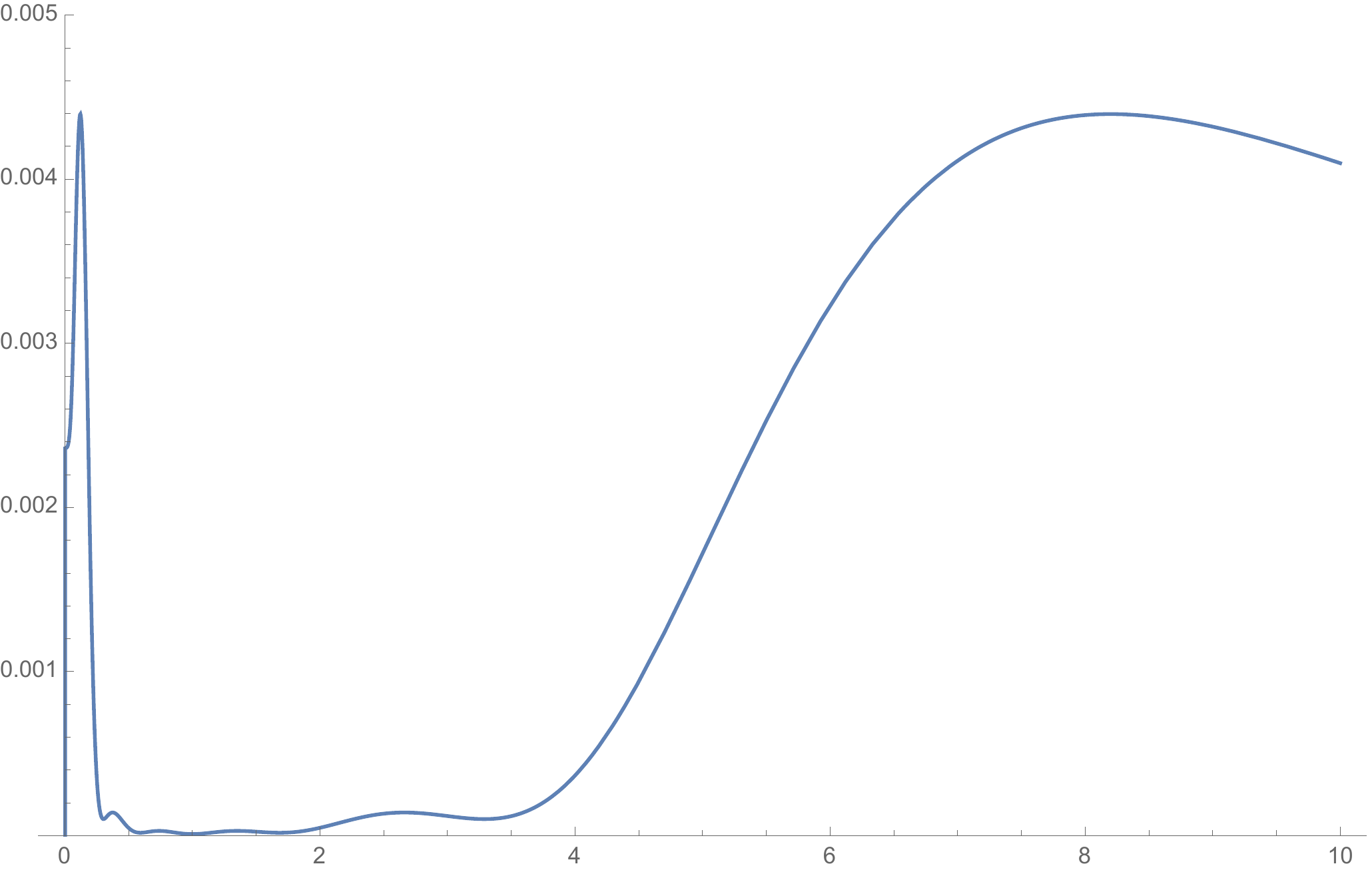}
  \caption{A plot of $\frac{\rho_{12}(\Omega(a))}{R(a)}-1$ as a function of $a$ when $\Omega(a)$ is the rectangle with side lengths $a$ and $1/a$.}
\end{figure}


\subsubsection{The House}  Here we consider a pentagon with vertices at $(-1,0), (1,0), (0,1-a), (1,a)$, and $(-1,a)$.  This shape has area $1$ and has a reflective symmetry around the $y$-axis (see Figure 2).
\begin{figure}[h!]\label{housepic}
  \centering
    \includegraphics[width=0.55\textwidth]{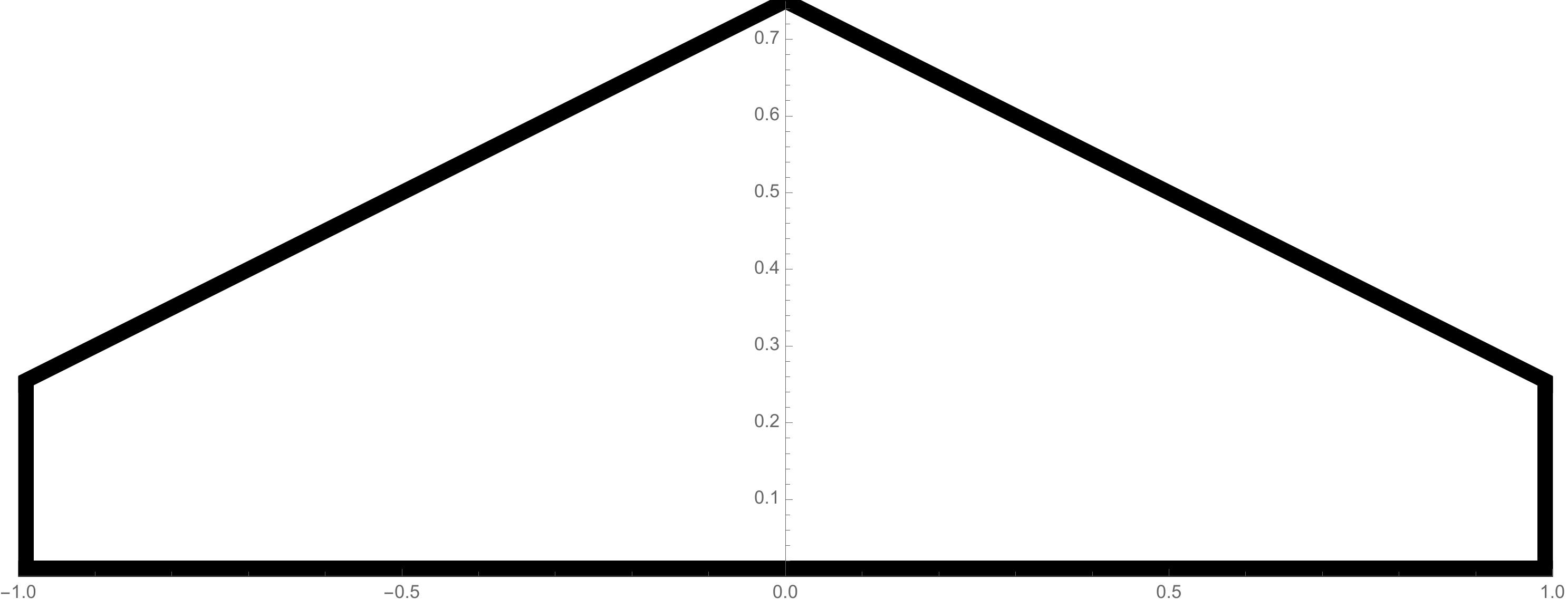}
  \caption{The house region when $a=1/4$.}
\end{figure}
When $a=0$, this shape is an isosceles triangle with base $2$ and altitude $1$ and when $a=1/2$, this shape is a rectangle with side lengths $2$ and $1/2$.  To employ the moments approach we discussed above, we need the following calculation.

\begin{prop}\label{housemom}
The moments of the house region described above with parameter $a\in[0,1/2]$ are given by
\begin{align*}
c_{n,m}&=\sum_{{0\leq j\leq n}\atop{0\leq k\leq m}}\bigg[\binom{n}{j}\binom{m}{k}(-1)^{m-k}i^{n+m-j-k}(1+(-1)^{j+k})\times\\
&\qquad\left(\frac{(1-a)^{m+n+1-j-k}\,_2F_1(j+k-m-n-1,j+k+1;j+k+2;2-\frac{1}{1-a})}{(j+k+1)(m+n+1-j-k)}\right)\bigg]
\end{align*}
\end{prop}

\begin{proof}
It is clear that
\begin{align*}
c_{n,m}&=\sum_{{0\leq j\leq n}\atop{0\leq k\leq m}}\bigg[\binom{n}{j}\binom{m}{k}(-1)^{m-k}i^{n+m-j-k}\times\\
&\left(\int_{-1}^0\int_{0}^{1-a+(1-2a)x}x^{j+k}y^{n+m-j-k}dydx+\int_{0}^{1}\int_{0}^{1-a-(1-2a)x}x^{j+k}y^{n+m-j-k}dydx\right)\bigg]
\end{align*}
so we need only evaluate the integrals.  We will show the calculation for evaluating the second integral and note that the first one is evaluated similarly.

Clearly this calculation reduces to evaluating
\[
\int_{0}^{1}x^{j+k}(1-a-(1-2a)x)^{n+m+1-j-k}dx=(1-a)^{n+m+1-j-k}\int_{0}^{1}x^{j+k}\left(1-\frac{1-2a}{1-a}x\right)^{n+m+1-j-k}dx.
\]
By Euler's integral formula for the Gauss Hypergoemetric function (see \cite[Section 1.6]{KLS}), we conclude that this is equal to
\[
\frac{(1-a)^{n+m+1-j-k}}{j+k+1}\,_2F_1\left(j+k-n-m-1,j+k+1;j+k+2;\frac{1-2a}{1-a}\right)
\]
as desired.
\end{proof}

Using this formula for the moments, we have employed the moment method to estimate the torsional rigidity of the house region using polynomials of degree at most $7$ (that is, we calculate $\rho_7(\Omega)$) and plotted the results in the figure.  The figure also shows a lower bound on torsional rigidity, which was obtained using (\ref{rhodef}) and the functions
\begin{align*}
u_1(x,y)&=y(2-y)(1-x^8)(y^2-(1-a-(1-2a)x)^2)(y^2-(1-a+(1-2a)x)^2)\\
u_2(x,y)&=y(1-x^4)(y^2-(1-a-(1-2a)x)^2)(y^2-(1-a+(1-2a)x)^2)\\
u_3(x,y)&=
\begin{cases}
y(1-y)(1-x^4)(y^2-(1-a-(1-2a)x)^2)  \qquad & \mbox{ if } x\leq0\\
y(1-y)(1-x^4)(y^2-(1-a+(1-2a)x)^2) & \mbox{ if } x\geq0.
\end{cases}
\end{align*}
Although the function $u_3$ may not be differentiable along the line $x=0$, it is clear that this function can be approximated uniformly by smooth functions in such a way that the lower bound obtained by $u_3$ is still valid.

Among these three trial functions, the function $u_1$ yields the best estimate for values of $a$ less than some critical value (which is approximately $1/3$), the function $u_2$ yields the best estimate when $a$ is between this critical value and a second critical value (which is approximately $2/5$), and the function $u_3$ yields the best estimate between this second critical value and $1/2$.  The best possible lower bound one can obtain from these calculations results from taking the maximum of the estimates given by $u_1$, $u_2$, and $u_3$, which is plotted in Figure 3.  The plot clearly suggests that the torsional rigidity of the house shaped region decreases as $a$ increases in $[0,1/2]$, which would have been difficult to predict.  Furthermore, the decrease seems to be very nearly linear in the parameter $a$.

\begin{figure}[h!]\label{houseplot}
  \centering
    \includegraphics[width=0.55\textwidth]{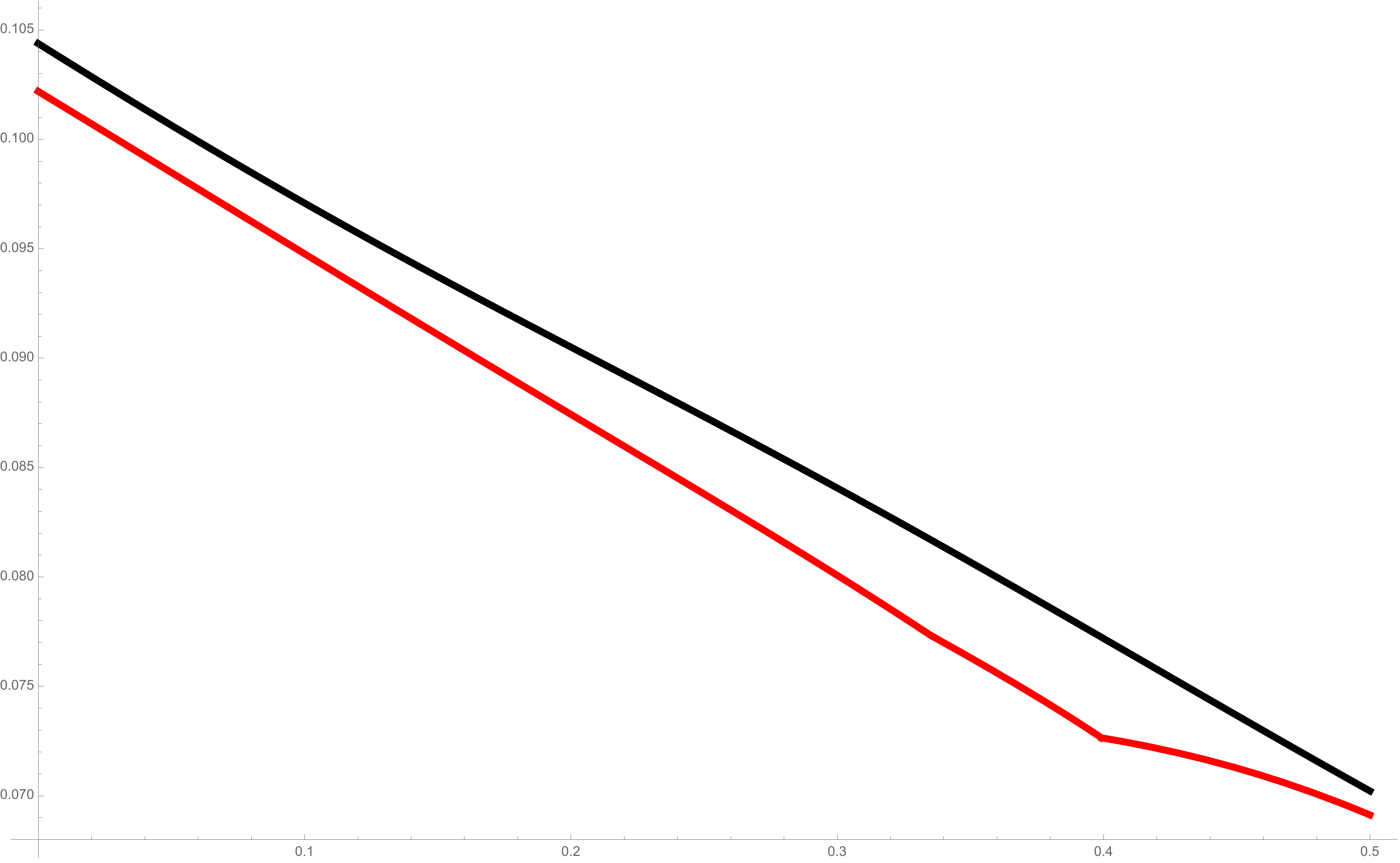}
  \caption{An upper and lower bound on $\rho(\Omega)$ as a function of $a$ when $\Omega$ is the house.}
\end{figure}

The upper bound that we obtain when $a=1/2$ is $0.0703208...$.  We know from (\ref{rectangle}) that the torsional rigidity of the rectangle with side lengths $1/2$ and $2$ is approximately $0.0702032...$, which indicates that the plot of the upper bound is close to the actual value, at least near that endpoint.  When $a=0$, this region is an isosceles right triangle with legs of length $\sqrt{2}$.  This is a special case of the example that we will study next.

\subsubsection{The Right Triangle}If we apply Theorem \ref{deg1} to the isosceles right triangle with vertices $(-\sqrt{2}/3,0)$, $(1/\sqrt{2}-\sqrt{2}/3,1/\sqrt{2})$, and $(1/\sqrt{2}-\sqrt{2}/3,-1/\sqrt{2})$ (which has centroid zero), then we get
\[
\rho(\Omega)\leq\rho_1(\Omega)=\frac{1}{24}.
\]
If we apply Theorem \ref{deg2} to the same triangle (which has $I_{21}=0$), we get
\[
\rho(\Omega)\leq\rho_2(\Omega)=\frac{11}{408}\approx0.0269608...
\]
Recall the formula given on \cite[page 108]{PS}, which tells us that
\begin{align}\label{triangle}
\rho(\Omega)=\frac{2^{10}}{\pi^6}\sum_{m,n=1}^{\infty}\frac{m^2}{(2n-1)^2[4m^2-(2n-1)^2][16m^4-(2n-1)^4]}\approx0.0260897...
\end{align}

We can apply the moments approach to the general right triangle with area $1$ having vertices at $(0,0)$, $(a,0)$, and $(a,2/a)$ for some $a>0$.  The symmetry of other known extremal shapes leads us to suspect that in this family of right triangles, the one with maximal torsional rigidity is the isosceles right triangle.  We formalize this in the following conjecture.

\begin{conj}\label{tri}
Among all right triangles with area $A>0$, the one with largest torsional rigidity is the isosceles right triangle with area $A$.
\end{conj}

Here we provide some strong numerical evidence supporting this conjecture and verify its accuracy to several decimal places.  Clearly it suffices to only consider right triangles with area $1$, which we parametrize as above and refer to the corresponding right triangle as $\Omega_a$.  It is easy to verify that the area moments of this triangle are
\[
c_{n,m}=\sum_{{0\leq j\leq n}\atop{0\leq k\leq m}}\binom{n}{j}\binom{m}{k}(-1)^{m-k}i^{n+m-j-k}\frac{2^{1+m+n-j-k}a^{2j+2k-m-n}}{(m+n+2)(m+n+1-j-k)}
\]
For any $n\in\bbN$, we can find the optimal polynomial approximant of degree at most $n$ to $\bar{z}$ in $L^2(\Omega_a,dA)$ using the moments approach.  We performed this calculation to find $\rho_{10}(\Omega_a)$ and the results are plotted in Figure 4.

\begin{figure}[h!]\label{triplot}
  \centering
    \includegraphics[width=0.55\textwidth]{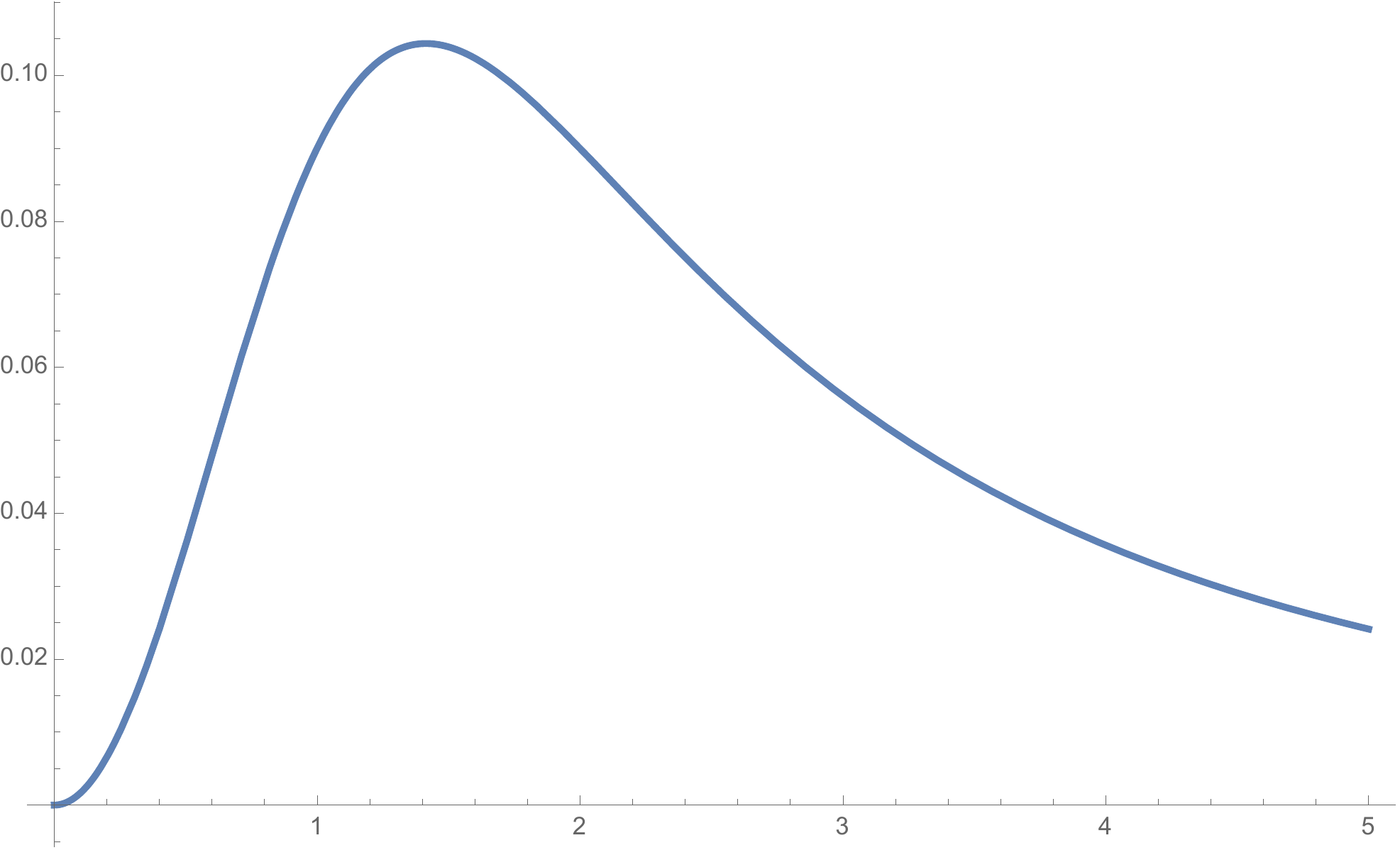}
  \caption{The approximate torsional rigidity $\rho_{10}(\Omega_a)$ as a function of $a$.}
\end{figure}

To use this information to support the conjecture, we again recall that we know the value of $\rho(\Omega_{\sqrt{2}})\approx(\sqrt{2})^4(0.0260897...)=0.1043586...$.  This means that $\max_{a>0}\{\rho(\Omega_a)\}$ can only be obtained by a value of $a$ for which $\rho_{10}(\Omega_a)\geq\rho(\Omega_{\sqrt{2}})$.  Looking closely at our graph, we see that this only happens for values of $a$ between $1.408131$ and $1.4203223$ (see Figure 5).  
\begin{figure}[h!]\label{triplot2}
  \centering
    \includegraphics[width=0.55\textwidth]{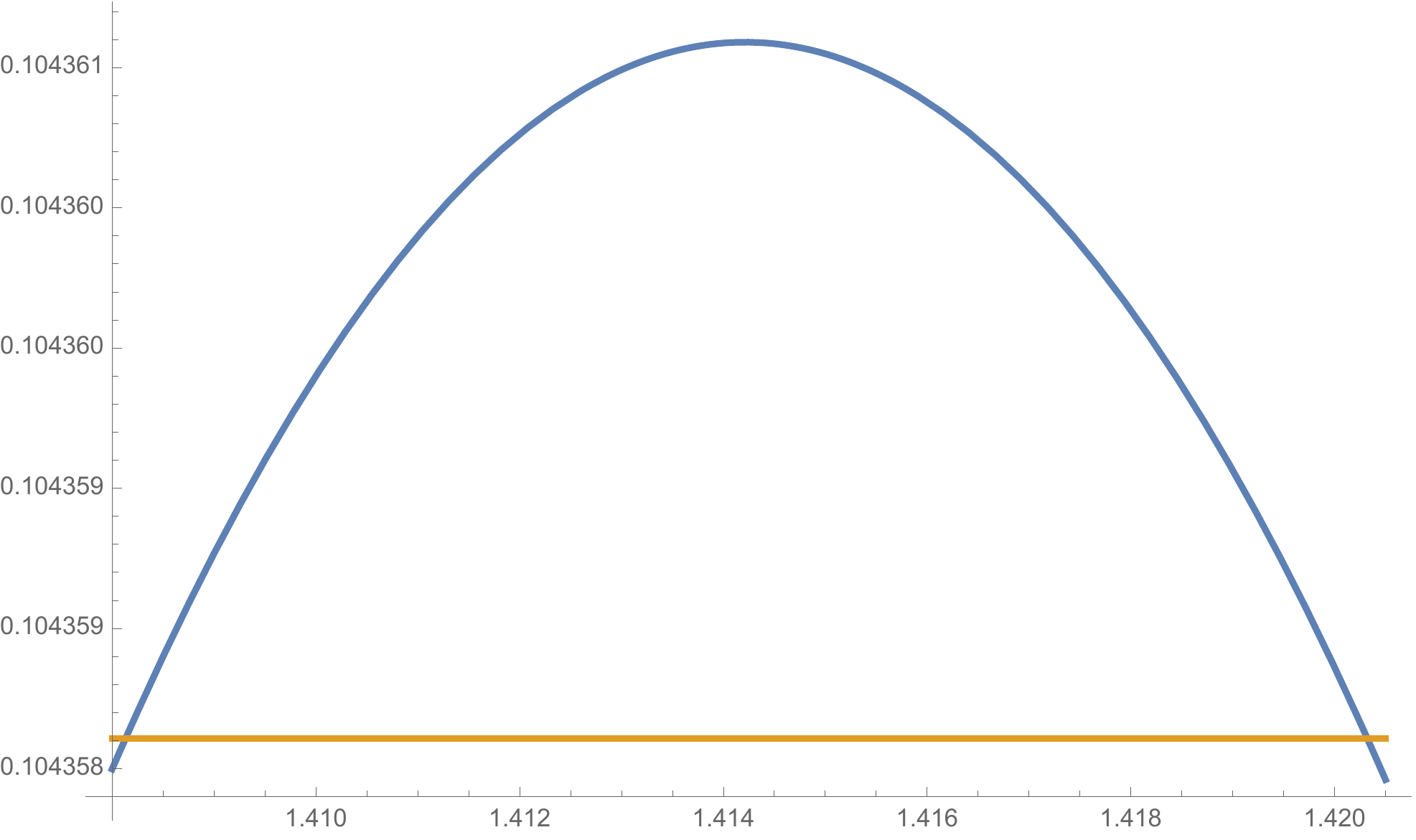}
  \caption{The approximate torsional rigidity $\rho_{10}(\Omega_a)$ as a function of $a$ for $a$ very close to $\sqrt{2}$.}
\end{figure}
Since $\sqrt{2}\approx1.41421356...$, we have numerically verified Conjecture \ref{tri} to an accuracy of approximately $0.0061$, which is slightly less than one half of one percent of $\sqrt{2}$.


\subsection{Rate of Convergence}  We have seen that the sequence $\{\rho_N(\Omega)\}_{N=1}^{\infty}$ gives a collection of upper bounds for $\rho(\Omega)$ that converges to $\rho(\Omega)$ as $\bnri$.  Let us say what we can about the rate of convergence.  To do so, one needs to know something about the asymptotics of the Bergman orthogonal polynomials $\{p_n\}_{n=0}^{\infty}$.    For this purpose, we consider the case in which $\Omega$ has analytic boundary and recall the notation and information from Section \ref{bergpoly}, specifically equation (\ref{pasy}).  If we define
\[
\eta^{\sharp}(z):=\overline{\eta\left(\frac{1}{\bar{z}}\right)},
\]
then $\eta^{\sharp}$ is univalent on the punctured disk of radius $1/r$ centered at $0$ and has a simple pole at $0$.  We then conclude that there is some $R\in(r,1)$ such that as $\nri$ it holds that
\begin{align*}
\langle w,p_n(w)\rangle&=\int_{\Omega}wp_n(w)dA(w)=\frac{1}{2i}\int_{\partial\Omega}|z|^2p_n(z)dz=\frac{\sqrt{n+1}}{2i\sqrt{\pi}}\int_{\partial\Omega}|z|^2\gamma(z)^n\gamma'(z)dz+\mco(R^n)\\
&=\frac{\sqrt{n+1}}{2i\sqrt{\pi}}\int_{\partial\bbD}|\eta(z)|^2z^ndz+\mco(R^n)=\frac{\sqrt{n+1}}{2i\sqrt{\pi}}\int_{\partial\bbD}\eta(z)\eta^{\sharp}(z)z^ndz+\mco(R^n)\\
&=\frac{\sqrt{n+1}}{2i\sqrt{\pi}}\int_{\{z:|z|=s\}}\eta(z)\eta^{\sharp}(z)z^ndz+\mco(R^n),
\end{align*}
where $s\in(r,R)$.  Observe that we used Green's Theorem in the second equality in this calculation.  This last integral clearly decays exponentially as $\nri$, so we have proven that the approximations $\rho_n(\Omega)$ converge to $\rho(\Omega)$ exponentially quickly in this case as $\nri$.  We state this result precisely in the following theorem.

\begin{theorem}\label{expconv}
If $\Omega$ is a Jordan region with analytic boundary, then
\[
\limsup_{\nri}\left(\rho_n(\Omega)-\rho(\Omega)\right)^{1/n}<1.
\]
\end{theorem}


\medskip

Now we turn our attention to a different approach to calculating torsional rigidity, which is more explicit but requires more specific information about the region $\Omega$.

\subsection{The Poisson Kernel Approach}\label{pkapproach}

Throughout this section we will let $\varphi:\Omega\rightarrow\bbD$ be a conformal bijection and let $\psi:\bbD\rightarrow\Omega$ be the inverse map.  We will give an explicit formula for the stress function of $\Omega$ and the Bergman projection of $\bar{z}$ in $A^2(\Omega)$, both of which require knowledge of the conformal map $\psi$.  The ideas of this method are not new and were presented in \cite[Section 134]{Mush}.  However, in \cite{Mush} there was no mention of the Bergman projection of $\bar{z}$.  Therefore, we present here a unified exposition that includes this additional information.

Define the function $F:\bbD\rightarrow\{z:\Real[z]>0\}$ by
\[
F(z):=\int_0^{2\pi}\frac{\eitheta+z}{\eitheta-z}\,\frac{|\psi(\eitheta)|^2}{2}\frac{d\theta}{2\pi},\qquad\qquad |z|<1.
\]

\begin{theorem}\label{fform}
The Bergman Projection of $\bar{z}$ and stress function for $\Omega$ are given by
 \begin{equation}\label{projform1}
Q(z)=F'(\varphi(z))\varphi'(z),\qquad\qquad \nu(z)=\Real[F(\varphi(z))]-\frac{|z|^2}{2}
\end{equation}
respectively.
\end{theorem}

\begin{proof}
We know from \cite[Section 1.3]{OPUC1} that $\Real[F(\eitheta)]=\frac{1}{2}|\psi(\eitheta)|^2$ Lebesgue almost everywhere.  Continuity of $\psi$ up to and including the boundary of $\Omega$ implies that this equality actually holds everywhere on $\partial\Omega$.  Therefore, $F(\varphi(z))+\overline{F(\varphi(z))}=|z|^2$ for every $z\in\partial\Omega$.  It follows from \cite[Theorem 1]{FK} that $Q$ has the desired form.  From this, it easily follows that $\nu$ as given in the statement of the Theorem solves the boundary value problem (\ref{bproblem}) and so it must be the stress function.
\end{proof}

It follows from Theorem \ref{fform} that the torsional rigidity of $\Omega$ is equal to
\begin{align}
\nonumber\rho(\Omega)=\int_{\Omega}|\bar{z}-F'(\varphi(z))\varphi'(z)|^2dA(z)&=\int_\Omega|z|^2dA(z)-\int_\Omega|F'(\varphi(z))\varphi'(z)|^2dA(z)\\
\label{rhofform}&=\int_{\bbD}|\psi(z)\psi'(z)|^2dA(z)-\int_\bbD|F'(z)|^2dA(z).
\end{align}
(compare with \cite[Equation 9]{DW}).

Let $\Omega_R$ be the image of $\{z:|z|<R\}$ under the map $\psi$.  Without further hypotheses, one can only define $\Omega_R$ for $R\leq1$.  However, if $\Omega$ has analytic boundary, then we can define $\Omega_R$ for some values of $R$ in $(1,\infty)$.  Let $R_{\Omega}$ be the largest value of $R\in[1,\infty]$ for which $\psi$ extends univalently to the disk centered at $0$ of radius $R$.

\begin{prop}\label{anacont}
The Bergman projection of $\bar{z}$ is analytic in $\Omega_{R_\Omega}$.
\end{prop}

\begin{proof}
We know that $\varphi$ is analytic in $\Omega_{R_{\Omega}}$ and we know from \cite[Theorem 7.1.2]{OPUC1} that $F$ is analytic in $\{z:|z|<R_{\Omega}\}$, so the result follows from Theorem \ref{fform}.
\end{proof}

Proposition \ref{anacont} shows that if $\psi$ can be univalently continued to a disk of radius larger than $1$, then the Bergman projection of $\bar{z}$ can be analytically continued outside of $\Omega$.  The following example shows that the converse is false.

\medskip

\noindent\textbf{Example: Equilateral Triangle}  Consider the equilateral triangle $T$ with vertices at $(1,0)$, $(-1/2,-\sqrt{3}/2)$, and $(-1/2,\sqrt{3}/2)$.  It is clear that the conformal bijection $\psi:\bbD\rightarrow T$ has three points on the boundary where the behavior is locally like $(z-z_0)^{1/3}$, and hence the radius of convergence of the Maclaurin series for $\psi$ is $1$.  It is easy to check that
\[
u(x,y):=\frac{2}{3}\left(x+\frac{1}{2}\right)\left(x-(1+\sqrt{3}y)\right)\left(x-(1-\sqrt{3}y)\right)=2\Real\left[\frac{z^3}{3}+\frac{1}{6}\right]-|z|^2
\]
(compare with \cite[Section 37]{Sok}).  Clearly $u(x,y)=0$ for all $x+iy\in\partial T$, so $\frac{z^3}{3}+\frac{1}{6}$ is an analytic function in $T$ whose real part is equal to $|z|^2/2$ everywhere on $\partial T$.  It follows from \cite[Theorem 1]{FK} that $z^2$ is the Bergman projection of $\bar{z}$ to the Bergman space of $T$.  This is an entire function, which shows that the converse to Proposition \ref{anacont} is false.



\medskip

Even if the conformal bijection $\psi$ is known explicitly, the calculations required to obtain the function $F$ can be difficult to complete or complete quickly.  Our next example shows that when $\psi$ is the reciprocal of a polynomial there is a faster way to obtain $F$ that uses the theory of orthogonal polynomials on the unit circle.

\medskip

\noindent\textbf{Example: Rational Conformal Map}  Consider the case when $\psi(z)=\sqrt{2}/p(z)$, where $p(z)$ is a degree $n$ polynomial with all of its zeros outside the closed unit disk and normalized so that
\[
\int_0^{2\pi}\frac{|\psi(\eitheta)|^2}{2}\frac{d\theta}{2\pi}=\int_0^{2\pi}\frac{1}{|p(\eitheta)|^2}\frac{d\theta}{2\pi}=1.
\]
In this case, one can avoid much calculation and take a shortcut to obtain the function $F$.

Consider $d\mu_p:=|p(\eitheta)|^{-2}\frac{d\theta}{2\pi}$, which is a probability measure on $[0,2\pi]$ (or equivalently, on the unit circle).  It is easy to show that the degree $n$ orthonormal polynomial for this measure is $p^*(z)=z^n\overline{p(1/\bar{z})}$ (notice that $p^*$ has degree $n$ because $p(0)\neq0$). While this is a standard fact from the theory of orthogonal polynomials on the unit circle, we present here a simple proof.  Indeed, let $W$ be any monic polynomial of degree $n$ and let $\kappa_n$ be the leading coefficient of $p^*$.  Since the integrand is subharmonic, we have
\[
\int_{0}^{2\pi}\left|\frac{W(\eitheta)}{p(\eitheta)}\right|^2\frac{d\theta}{2\pi}=\int_{0}^{2\pi}\left|\frac{W(\eitheta)}{p^*(\eitheta)}\right|^2\frac{d\theta}{2\pi}\geq\frac{W(\infty)^2}{p^*(\infty)^2}=\frac{1}{\kappa_n^2},
\]
with equality if and only if the integrand in the second integral is constant.  Thus, the choice of $W$ that minimizes this integral is $p^*/\kappa_n$, so this must be the monic orthogonal polynomial of degree $n$ for $\mu_p$ and hence $p^*$ is the negree $n$ orthonormal polynomial for $\mu_p$ as desired.

Define $\Phi_n:=p^*/\kappa_n$. Recall our discussion from Section \ref{uc}, which implies we can use $\Phi_n$ to determine the monic degree $m$ orthogonal polynomials for $\mu_p$ for all $m\leq n$.  Let us call this collection $\{\Phi_m\}_{m=0}^n$.  By evaluating these polynomials at zero, we can determine the first $n$ Verblunsky coefficients for $\mu_p$; call them $\{\alpha_0,\ldots,\alpha_{n-1}\}$.  One can then iterate the Szeg\H{o} recursion algorithm (see Section \ref{uc}) using the collection $\{-\alpha_0,\ldots,-\alpha_{n-1}\}$ to obtain the monic degree $n$ second kind polynomial for $\mu_p$; call it $\Psi_n$.  From \cite[Theorem 3.2.4]{OPUC1}, it follows that
\[
F(z)=\frac{\overline{\Psi_n(1/\bar{z})}}{\overline{\Phi_n(1/\bar{z})}},\qquad\qquad |z|<1.
\]
One can then plug this into formula (\ref{rhofform}) to calculate the torsional rigidity of $\Omega$.

As a concrete example, let us consider the case
\[
\psi(z)=\frac{9\sqrt{2}}{\sqrt{11}(z+2)^3}
\]
(see Figure 6).
\begin{figure}[h!]\label{ratopic}
  \centering
    \includegraphics[width=0.35\textwidth]{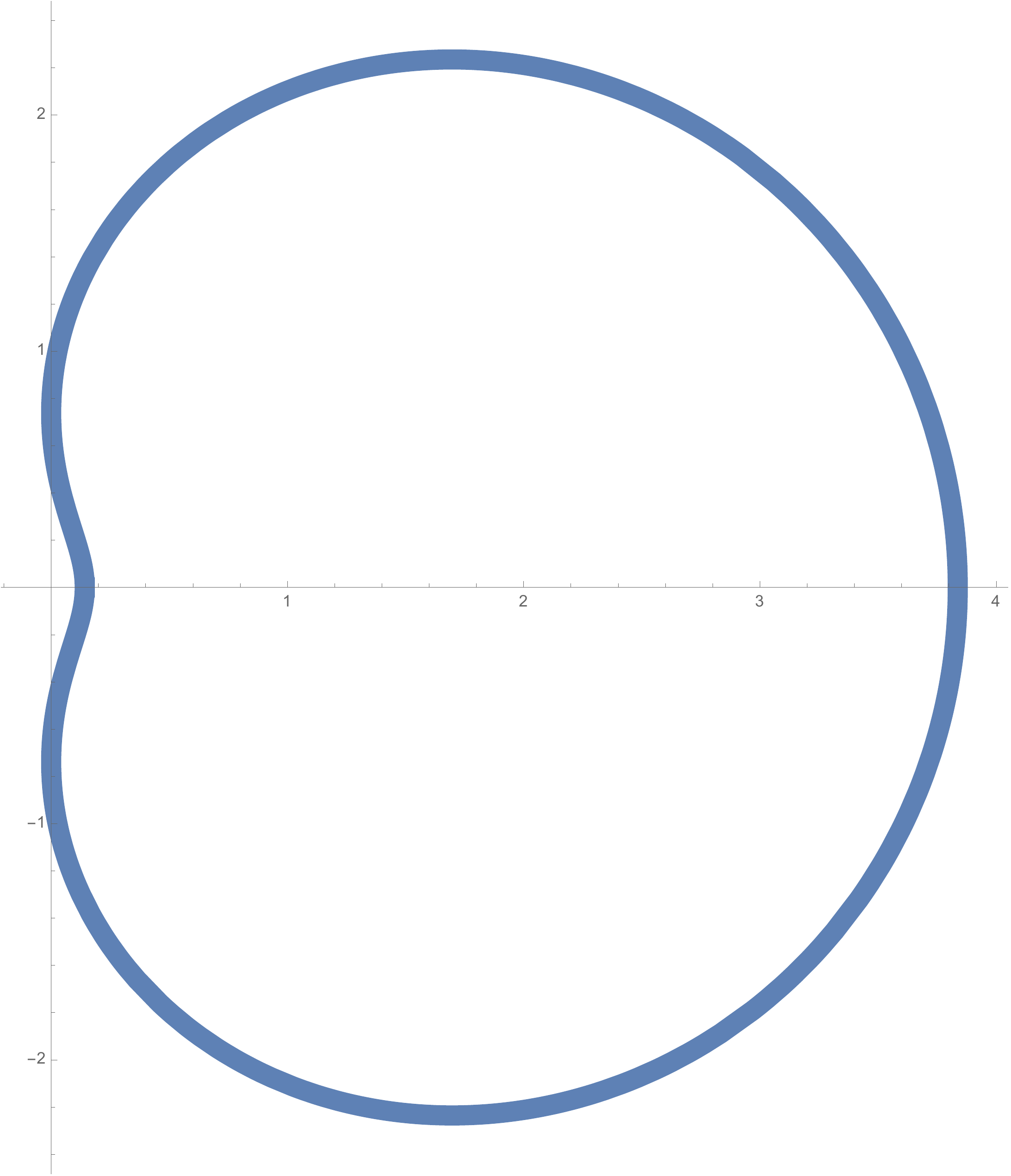}
  \caption{The boundary of the region $\psi(\bbD)$ when $\psi(z)=\frac{9\sqrt{2}}{\sqrt{11}(z+2)^3}$.}
\end{figure}
In this case, $p(z)=\sqrt{11}(z+2)^3/9$, so $p^*(z)=\frac{\sqrt{11}}{9}(1+2z)^3$ and hence $\Phi_3(z)=(z+1/2)^3$.  We then calculate
\[
-\bar{\alpha}_2=\Phi_3(0)=1/8.
\]
If we plug this into \cite[Equation 1.5.40]{OPUC1}, we get
\[
\Phi_2(z)=\frac{64}{63z}((z+1/2)^3-(1+z/2)^3/8)=z^2+\frac{10}{7}z+\frac{4}{7}.
\]
Repeating this procedure shows
\[
-\bar{\alpha}_1=\Phi_2(0)=4/7
\]
and hence
\[
\Phi_1(z)=\frac{49}{33z}\left(z^2+\frac{10}{7}z+\frac{4}{7}-\frac{4}{7}\left(\frac{4}{7}z^2+\frac{10}{7}z+1\right)\right)=z+\frac{10}{11}.
\]
We conclude that $\{\alpha_0,\alpha_1,\alpha_2\}=\{-10/11,-4/7,-1/8\}$.  Now, if we perform the Szeg\H{o} recursion with the set $\{10/11,4/7,1/8\}$, then we get
\[
\Psi_3(z)=z^3-\frac{7}{22}z^2-\frac{23}{24}z-\frac{1}{8}.
\]
Combining this with our formula for $\Phi_3(z)$, we find
\[
F(z)=\frac{24-\frac{84}{11}z-23z^2-3z^3}{3(z+2)^3}.
\]
Clearly
\[
\varphi(z)=\left(\frac{9\sqrt{2}}{\sqrt{11}z}\right)^{1/3}-2,
\]
so Theorem \ref{fform} allows us to write down explicit formulas for $Q(z)$ and $\nu(z)$.  We can also calculate the torsional rigidity using the formula (\ref{rhofform}).  Indeed, we have
\begin{align*}
\rho(\Omega)=\frac{1}{4}\int_{\bbD}\left|\frac{972}{11(z+2)^7}\right|^2dA(z)-\int_{\bbD}\left|\frac{55z^2-844z-960}{33(z+2)^4}\right|^2dA(z),
\end{align*}
and each of these integral can be easily evaluated using Taylor series.

\section{Further Examples}\label{egs}

Here we present two detailed calculations.  The first considers a ``dented disk" region and the second considers Neumann's oval.  In both cases, we can use the method described in Section \ref{pkapproach} to find exact expressions for the torsional rigidity, stress function, and Bergman projection of $\bar{z}$.

\subsection{The Dented Disk}\label{dent}

In this example, we use the Poisson Kernel approach to calculate the torsional rigidity of the region that is the image of the unit disk under the conformal bijection
\[
\psi(z)=z+\frac{a}{z-b}
\]
subject to the constraints that
\begin{itemize}
\item[(i)] $a\neq0$, $|b|>1$
\item[(ii)] $|b\pm\sqrt{a}|>1$
\item[(iii)] $\displaystyle{\left|b+\frac{a}{e^{it}-b}\right|>1}$ for all $t\in\bbR$
\end{itemize}
The second of these conditions assures us that $\psi'\neq0$ in $\bbD$.  The third condition assures us that $\psi$ is injective on $\bbD$.  Therefore, $\psi$ is a conformal bijection of $\bbD$ with its range (see Figure 7 for an example).

\begin{figure}[h!]\label{dentpic}
  \centering
    \includegraphics[width=0.35\textwidth]{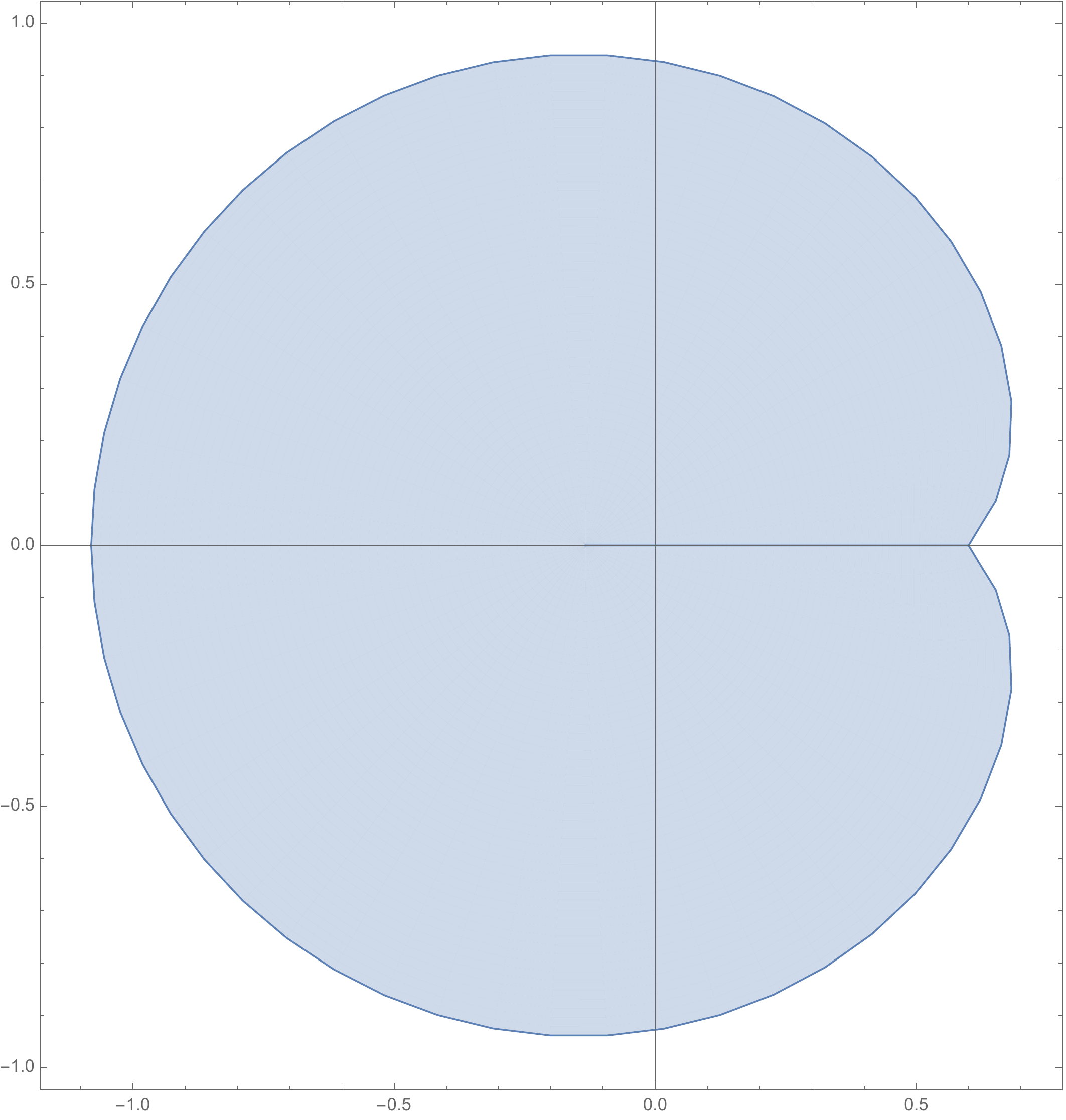}
  \caption{The region $\psi(\bbD)$ when $a=1/5$ and $b=3/2$.}
\end{figure}

By Taylor expanding the integrand in powers $\eitheta$, one can calculate
\[
h_j:=\int_0^{2\pi}e^{-ij\theta}\frac{|\psi(\eitheta)|^2}{2}\frac{d\theta}{2\pi}=
\begin{cases}
\frac{1}{2}\left(1-2\Real\left(\frac{a}{b^2}\right)+\frac{|a|^2}{|b|^2-1}\right)\qquad\qquad\qquad\qquad& j=0\\
\,\\
\frac{1}{2}\left(-\frac{a}{b^3}-\frac{\bar{a}}{\bar{b}}+\frac{|a|^2}{b(|b|^2-1)}\right) & j=1\\
\,\\
\frac{1}{2b^j}\left(-\frac{a}{b^2}+\frac{|a|^2}{(|b|^2-1)}\right) & j\geq2
\end{cases}
\]
Since
\begin{align}\label{fsum}
F(z)=h_0+2\sum_{j=1}^{\infty}h_jz^j,
\end{align}
we have
\begin{align*}
&F(z)\\
&\,=\frac{1}{2}\left(1-2\Real\left(\frac{a}{b^2}\right)+\frac{|a|^2}{|b|^2-1}\right)+\left(\frac{|a|^2}{b(|b|^2-1)}-\frac{a}{b^3}-\frac{\bar{a}}{\bar{b}}\right)z+\left(\frac{|a|^2}{(|b|^2-1)}-\frac{a}{b^2}\right)\sum_{j=2}^{\infty}\left(\frac{z}{b}\right)^j\\
&\,=\frac{1}{2}\left(1-2\Real\left(\frac{a}{b^2}\right)+\frac{|a|^2}{|b|^2-1}\right)+\left(\frac{|a|^2}{b(|b|^2-1)}-\frac{a}{b^3}-\frac{\bar{a}}{\bar{b}}\right)z-\left(\frac{|a|^2}{(|b|^2-1)}-\frac{a}{b^2}\right)\frac{z^2}{b(z-b)}\\
&\,=\frac{1}{2}\left(1-2\Real\left(\frac{a}{b^2}\right)+\frac{|a|^2}{|b|^2-1}\right)+\left(\frac{|a|^2}{b(|b|^2-1)}-\frac{a}{b^3}-\frac{\bar{a}}{\bar{b}}\right)z-\left(\frac{|a|^2}{(|b|^2-1)}-\frac{a}{b^2}\right)\frac{z^2\psi(z)-z^3}{ab}
\end{align*}
It is easy to check that
\[
\varphi(z)=\frac{b+z-\sqrt{b^2+z^2-4a-2bz}}{2}
\]
is the conformal bijection from $\Omega\rightarrow\bbD$ that is inverse to $\psi$, so one can explicitly write down the Bergman projection of $\bar{z}$ and the stress function of $\Omega$ using Theorem \ref{fform}.  Indeed, we have
\begin{align*}
Q(z)&=\left(\frac{|a|^2}{b(|b|^2-1)}-\frac{a}{b^3}-\frac{\bar{a}}{\bar{b}}\right)\varphi'(z)-\left(\frac{|a|^2}{(|b|^2-1)}\right)\frac{(z\varphi(z)^2-\varphi(z)^3)'}{ab},\\
\nu(z)&=\Real\bigg[\frac{1}{2}\left(1-2\Real\left(\frac{a}{b^2}\right)+\frac{|a|^2}{|b|^2-1}\right)+\left(\frac{|a|^2}{b(|b|^2-1)}-\frac{a}{b^3}-\frac{\bar{a}}{\bar{b}}\right)\varphi(z)\\
&\qquad\qquad\qquad\qquad\qquad\qquad\qquad-\left(\frac{|a|^2}{(|b|^2-1)}-\frac{a}{b^2}\right)\frac{z\varphi(z)^2-\varphi(z)^3}{ab}\bigg]-\frac{|z|^2}{2}
\end{align*}

To calculate the torsional rigidity, we use the formula (\ref{rhofform}).  We have
\begin{align*}
\rho(\Omega)&=\frac{1}{4}\int_{\bbD}\left|2z+\sum_{j=0}^{\infty}\left(\frac{(j+1)a}{b^{j+1}}\left(\frac{a(j+2)}{b^{2}}-2\right)\right)z^j\right|^2dA(z)\\
&\qquad\qquad-\int_{\bbD}\left|\left(\frac{|a|^2}{b(|b|^2-1)}-\frac{a}{b^3}-\frac{\bar{a}}{\bar{b}}\right)+\left(\frac{|a|^2}{(|b|^2-1)}-\frac{a}{b^2}\right)\sum_{j=1}^{\infty}\frac{(j+1)z^j}{b^j}\right|^2dA(z)\\
&=\pi\bigg(\left|\frac{a}{b}\left(\frac{a}{b^2}-1\right)\right|^2-\left(\left|\frac{|a|^2}{b(|b|^2-1)}-\frac{a}{b^3}-\frac{\bar{a}}{\bar{b}}\right|^2+\left|\frac{|a|^2}{(|b|^2-1)}-\frac{a}{b^2}\right|^2\cdot\frac{(2-|b|^{-2})}{(|b|^2-1)^2}\right)\bigg)\\
&\qquad+\frac{1}{4}\sum_{j=2}^{\infty}\frac{(j+1)|a|^2}{|b|^{2j+2}}\left|\frac{a(j+2)}{b^{2}}-2\right|^2+\frac{1}{2}\left|1+\frac{a(3ab^{-2}-2)}{b^2}\right|^2.
\end{align*}
Notice that if $b\rightarrow\infty$ or if $a\rightarrow0$, then the torsional rigidity approaches $\pi/2=\rho(\bbD)$, which is what we expect.  A similar phenomenon occurs if $a$ and $b$ tend to infinity together and in such a way that $a/b$ approaches a constant.



\subsection{Neumann's Oval}\label{oval}

This set was explored in an example in \cite{EL}.  Its boundary is described as all pairs $(x,y)$ so that
\[
(x^2+y^2)^2=a^2(x^2+y^2)+4x^2
\]
for some parameter $a>0$.  The conformal bijection from $\bbD$ to this region is given by
\[
\psi(z)=\frac{(R^4-1)z}{R(R^2-z^2)},
\]
where
\[
R=\frac{a+\sqrt{a^2+4}}{2}>1.
\]
With this information, one can compute the torsional rigidity and stress function.

First, we may write 
\[
\psi(z)=\frac{(R^{4}-1)z}{R\left(R^{2}-z^{2}\right)}=\frac{R^{4}-1}{R^{3}}\sum_{n=0}^{\infty}\frac{z^{2n+1}}{R^{2n}},
\]
If we define $h_j$ as in the previous example, then we use this sum to calculate 
\begin{align*}
h_{k}&=\frac{1}{2}\int_{0}^{2\pi}e^{-ik\theta}\left|\psi(e^{i\theta})\right|^{2}\frac{d\theta}{2\pi}=\frac{1}{2}\int_{0}^{2\pi}e^{-ik\theta}\psi(e^{i\theta})\overline{\psi(e^{i\theta})}\frac{d\theta}{2\pi}\\
&=\frac{(R^{4}-1)^{2}}{2R^{6}}\sum_{n=0}^{\infty}\sum_{m=0}^{\infty}\int_{0}^{2\pi}\left(\frac{e^{i(2n+1-k)\theta}}{R^{2n}}\right)\left(\frac{e^{-i(2m+1)\theta}}{R^{2m}}\right)\frac{d\theta}{2\pi}.
\end{align*}
When $k$ is odd, it is clear from orthogonality that $h_k=0$.  When $k=2j$, 
\[
\int_{0}^{2\pi}\left(\frac{e^{i(2n+1-k)\theta}}{R^{2n}}\right)\left(\frac{e^{-i(2m+1)\theta}}{R^{2m}}\right)\frac{d\theta}{2\pi}=\begin{cases}
0 & 2m+1\neq2n+1-k\\
\frac{1}{R^{4n-2j}} & m=n-j
\end{cases}.
\]
Summing over all $n$ from $j$ to $\infty$ shows
\[
h_{k}=\begin{cases}
\frac{R^{4}-1}{2R^{2j+2}} & k=2j\\
0 & k=2j+1
\end{cases}\qquad j=0,1,2,\ldots.
\]
We conclude from (\ref{fsum}) that
\begin{align*}
F(z)&=F(0)+2\sum_{k=1}^{\infty}h_{k}z^{k}=\frac{R^4-1}{2R^2}+\frac{(R^4-1)}{R^2}\sum_{j=1}^{\infty}\frac{z^{2j}}{R^{2j}}\\
&=\frac{R^4-1}{2R^2}+\frac{(R^4-1)z^2}{R^2(R^2-z^2)}=\frac{R^4-1}{2R^2}+\frac{z\psi(z)}{R}.
\end{align*}

Now, 
\[
\psi(z)\psi'(z)=\frac{\left(R^{4}-1\right)^{2}}{R^{2}}\left(\frac{R^{2}z+z^3}{\left(R^{2}-z^{2}\right)^{3}}\right)=\frac{\left(R^{4}-1\right)^{2}}{R^{6}}\left[\frac{z}{\left(1-\frac{z^{2}}{R^{2}}\right)^{3}}+\frac{z^{3}}{R^{2}\left(1-\frac{z^{2}}{R^{2}}\right)^{3}}\right].
\]
Since
\[
\frac{z}{\left(1-\frac{z^{2}}{R^{2}}\right)^{3}}=\frac{d}{dz}\frac{R^2/4}{\left(1-\frac{z^{2}}{R^{2}}\right)^{2}}=\frac{d}{dz}\frac{R^{2}}{4}\sum_{n=0}^{\infty}\left(n+1\right)\left(\frac{z^{2}}{R^{2}}\right)^{n}=\frac{R^{2}}{4}\sum_{n=0}^{\infty}2n\left(n+1\right)\frac{z^{2n-1}}{R^{2n}},
\]
we may write 
\begin{align*}
\psi(z)\psi'(z)&=\frac{\left(R^{4}-1\right)^{2}}{R^{6}}\left[\frac{R^{2}}{4}\sum_{n=0}^{\infty}2\left(n+1\right)\left(n+2\right)\frac{z^{2n+1}}{R^{2\left(n+1\right)}}+\frac{1}{4}\sum_{n=0}^{\infty}2n\left(n+1\right)\frac{z^{2n+1}}{R^{2n}}\right]\\
&=\frac{\left(R^{4}-1\right)^{2}}{R^{6}}\sum_{n=0}^{\infty}(n+1)^2\frac{z^{2n+1}}{R^{2n}}=\frac{(R^4-1)^2z(R^2+z^2)}{R^2(R^2-z^2)^3}.
\end{align*}
Thus we can compute the Bergman norm of $\psi(z)\psi'(z)$ as 
\[
\left\Vert \psi(z)\psi'(z)\right\Vert _{\mathcal{B}}^{2}=\frac{\left(R^{4}-1\right)^{4}}{2R^{12}}\sum_{n=0}^{\infty}\frac{(n+1)^{3}}{R^{4n}}.
\]
Using this and \eqref{rhofform}, we compute torsional rigidity as 
\begin{align*}
\rho(\Omega)&=\frac{\pi\left(R^{4}-1\right)^{4}}{2R^{12}}\sum_{n=0}^{\infty}\frac{(n+1)^{3}}{R^{4n}}-\pi\sum_{j=0}^{\infty}2j\frac{(R^{4}-1)^2}{R^{4j+4}}=\pi\left(\frac{R^4+R^{-4}}{2}\right)=\pi\left(\frac{a^4}{2}+2a^2+1\right).
\end{align*}
Notice that this grows like $\pi R^4/2$ as $R\rightarrow\infty$, which is the torsional rigidity of the disk of radius $R$.

It is easy to check that
\[
\varphi(z)=\frac{1-R^4+\sqrt{(R^4-1)^2+4R^4z^2}}{2Rz}
\]
and hence we can use Theorem \ref{fform} and our formula for $F(z)$ to write the stress function
\begin{equation*}
\nu(z)=\frac{R^4-1}{2R^2}+\frac{\Real[z\varphi(z)]}{R}-\frac{|z|^2}{2}
\end{equation*}
and the Bergman projection of $\bar{z}$ is 
\begin{equation*}
Q(z)=\frac{\varphi(z)+z\varphi'(z)}{R}.
\end{equation*}

\vspace{7mm}


\vspace{8mm}

\noindent Matthew Fleeman

\noindent\small{Matthew$\_$Fleeman@baylor.edu}

\bigskip

\noindent Brian Simanek

\noindent\small{Brian$\_$Simanek@baylor.edu}

\end{document}